\documentclass[11pt]{article}
\usepackage{amsfonts,amsmath, amssymb,latexsym,mathrsfs,upgreek}
\usepackage{multirow}
\usepackage{verbatim}
\usepackage{longtable}
\usepackage{hyperref}
\hypersetup{colorlinks=true,linkcolor=blue, linktocpage}
\usepackage{tocloft}
\usepackage{cancel}

\usepackage{tikz}
\usetikzlibrary{matrix,arrows,decorations.pathreplacing}

\usepackage{todonotes}
\usepackage{graphicx}

\usepackage{amsthm}
\usepackage{mathtools}
\usepackage{manfnt}
\usepackage{enumitem}
\usepackage{setspace}
\usepackage{calc,footnote}
\usepackage{graphicx}
\usepackage{tikz}
\usepackage{tikz-cd}
\usepackage{comment}

\newtheorem{theorem}{Theorem}[section]
\newtheorem{proposition}[theorem]{Proposition}
\newtheorem{lemma}[theorem]{Lemma}

\newtheorem{conjecture}[theorem]{Conjecture}

\theoremstyle{definition}

\renewenvironment{proof}{\noindent {\bf Proof:}}{$\Box$ \vspace{2 ex}}
\newtheorem{remark}[theorem]{Remark}

\newtheorem*{theorem*}{Theorem}
\newtheorem*{proposition*}{Proposition}
\newtheorem*{lemma*}{Lemma}

\theoremstyle{remark}

\newcommand{\defeq}{\vcentcolon=}

\newcommand\nc{\newcommand}
\nc{\on}{\operatorname}
\nc\renc{\renewcommand}
\nc{\BR}{\mathbb R}
\nc{\BC}{\mathbb C}
\nc{\BQ}{\mathbb Q}
\nc{\BZ}{\mathbb Z}
\nc{\BN}{\mathbb N}
\nc{\BP}{\mathbb P}
\nc{\BA}{\mathbb A}
\nc{\Hom}{\on{Hom}}
\nc{\wt}{\widetilde}
\nc{\vspan}{\on{span}}
\nc{\ord}{\on{ord}}
\nc{\im}{\on{im}}
\nc{\Mat}{\on{Mat}}
\nc{\can}{\on{can}}
\nc{\coker}{\on{coker}}
\nc{\ev}{\on{ev}}
\nc{\Tr}{\on{Tr}}
\nc{\End}{\on{End}}
\nc{\swap}{\on{swap}}
\nc{\Set}{\on{Set}}
\nc{\bC}{{\bm C}}
\nc{\bc}{{\bm c}}
\nc{\bD}{{\bm D}}
\nc{\bd}{{\bm d}}
\nc{\bE}{{\bm E}}
\nc{\be}{{\bm e}}
\nc{\bff}{{\bm f}}
\nc{\Bf}{\mathbb{f}}

\nc{\adj}{\on{adj}}
\nc{\tensor}[3]{#1 \underset{#2}\otimes #3}
\nc{\Nat}{\on{Nat}}
\nc{\op}{\on{op}}
\nc{\funct}{\on{funct}}
\nc{\Ob}{\on{Ob}}
\nc{\fR}{\mathfrak{R}}
\nc{\Vect}{\on{Vect}}
\nc{\ns}{\on{non-spec}}
\nc{\ol}{\overline}
\nc{\ul}{\underline}
\nc{\w}{\omega}
\nc{\nlog}{\on{nlog}}
\nc{\Aut}{\on{Aut}}
\nc{\Gal}{\on{Gal}}
\nc{\Poss}{\on{Poss}}
\nc{\ksep}{\on{sep}}
\nc{\low}{\on{low}}
\nc{\Stab}{\on{Stab}}
\nc{\pp}{\mathfrak{p}}
\nc{\OO}{\mathcal{O}}
\nc{\mm}{\mathfrak{m}}
\nc{\qq}{\mathfrak{q}}
\nc{\Nm}{\on{Nm}}
\nc{\Ann}{\on{Ann}}
\nc{\gug}{\mathfrak{g}}
\nc{\hug}{\mathfrak{h}}
\nc{\mf}{\mathfrak}
\nc{\mc}{\mathcal}
\nc{\Sym}{\on{Sym}}
\renc{\O}{\mc{O}}
\nc{\al}{\alpha}

\def\Z{{\mathbb Z}}

\def\End{{\rm End}}
\def\Sym{{\rm Sym}}
\def\Sel{{\rm Sel}}

\def\PGL{{\rm PGL}}

\def\ls{{\rm ls}}

\def\ur{{\rm ur}}
\def\mon{{\rm mon}}
\def\Stab{{\rm Stab}}
\def\Sym{{\rm Sym}}
\def\Jac{{\rm Jac}}
\def\Cl{{\rm Cl}}
\def\O{{\mathcal O}}
\def\cO{{\mathcal O}}

\def\PSO{{\rm PSO}}

\def\Im{{\rm Im}}
\def\adj{{\rm adj}}
\def\Aut{{\rm Aut}}
\def\irr{{\rm irr}}

\def\Vol{{\rm Vol}}
\def\Avg{{\rm Avg}}
\def\R{{\mathbb R}}

\def\FF{{\mathcal F}}

\def\cM{{\mathcal M}}

\def\cS{{\mathcal S}}
\def\EE{{\mathcal E}}

\def\Q{{\mathbb Q}}

\def\J{{\mathcal J}}
\def\cA{{\mathcal A}}

\def\cI{{\mathcal I}}

\def\Z{{\mathbb Z}}

\def\Q{{\mathbb Q}}
\def\C{{\mathbb C}}

\def\Res{{\rm{Res}}}

\makeatletter
\def\@tocline#1#2#3#4#5#6#7{\relax
  \ifnum #1>\c@tocdepth
  \else
    \par \addpenalty\@secpenalty\addvspace{#2}
    \begingroup \hyphenpenalty\@M
    \@ifempty{#4}{
      \@tempdima\csname r@tocindent\number#1\endcsname\relax
    }{
      \@tempdima#4\relax
    }
    \parindent\z@ \leftskip#3\relax \advance\leftskip\@tempdima\relax
    \rightskip\@pnumwidth plus4em \parfillskip-\@pnumwidth
    #5\leavevmode\hskip-\@tempdima
      \ifcase #1
       \or\or \hskip 1em \or \hskip 2em \else \hskip 3em \fi
      #6\nobreak\relax
    \dotfill\hbox to\@pnumwidth{\@tocpagenum{#7}}\par
    \nobreak
    \endgroup
  \fi}
\makeatother

 \allowdisplaybreaks

\begin{document}

\title{\vspace*{-35pt}The second moment of the size of the $2$-class group \\ of monogenized cubic fields}

\author{Manjul Bhargava, Arul Shankar, and Ashvin A.~Swaminathan}

\maketitle

\begin{abstract}
We prove that when totally real (resp., complex) monogenized cubic number fields are ordered by height, the second moment of the size of the $2$-class group is at most $3$ (resp., at most $6$). In the totally real case, we further prove that the second moment of the size of the narrow $2$-class group is at most $9$. This result gives further evidence in support of the general observation, first made in work of Bhargava--Hanke--Shankar and recently formalized into a set of heuristics in work of Siad--Venkatesh, that monogenicity has an altering effect on class group distributions. All of the upper bounds we obtain are tight, conditional on tail estimates.
\end{abstract}

%

\section{Introduction}

In~\cite{BSHpreprint}, Bhargava, Hanke, and Shankar made a surprising observation concerning the distribution of the $2$-class group of cubic number fields: they proved that when one passes to the thin subfamily of monogenized fields---i.e., fields $K$ whose rings of integers $\mc{O}_K$ are of the form $\mc{O}_K = \Z[\theta]$, where $\theta \in \mc{O}_K$ is a specified generator---the \emph{average} size of the $2$-class group increases, relative to the corresponding average for the full family of cubic fields; see Theorem~\ref{thm-prev} for a precise formulation. Their result initiated a program of research in which one seeks to understand the impact of monogenicity and other global arithmetic conditions on class group distributions. For instance, it was later shown by Shankar, Siad, and Swaminathan~\cite{SSSpreprint} that when one passes to the even thinner subfamily of unit-monogenized cubic fields---i.e., fields $K$ whose rings of integers $\mc{O}_K$ can be expressed as $\mc{O}_K = \Z[\theta]$, where $\theta \in \mc{O}_K^\times$ is a specified generator---the average size of the $2$-class group increases even further. In a different direction, Siad~\cite{Siadthesis1} proved that a similar increasing effect occurs in the average size of the $2$-class group for monogenized number fields of any odd degree $n \geq 3$.

The purpose of this paper is to study the \emph{second moment} of the size of the $2$-class group in the family of monogenized cubic number fields. Specifically, we prove that the second moment is bounded, and conditional on a tail estimate, we prove that the bound we obtain is optimal; see Theorem~\ref{thm-main} for a precise formulation. This result is the first of its kind toward determining the second moment of a class group distribution for number fields of degree greater than two, and it provides further evidence of the observation that monogenicity has an increasing effect on the $2$-class groups of number fields.

\subsection{Monogenized cubic number fields}

In this subsection, we state precisely what it means for a cubic number field to be monogenized, and we introduce the height function by which we order such fields.

\subsubsection{The definition} \label{sec-mondef}

Let $R$ be a principal ideal domain with fraction field $F$. Recall that an \'{e}tale $F$-algebra $K$ is said to be {\bf monogenic} if its ring of integers $\mc{O}_K$ is singly generated as a $\Z$-algebra, i.e., $\O_K=\Z[\alpha]$ for some $\alpha\in\O_K$; such an $\alpha$
is then called a {\bf monogenizer} of $K$ or of $\O_K$. A {\bf monogenized \'{e}tale algebra} consists of the data $(K,\alpha)$ of a monogenic \'{e}tale algebra $K$ along with a specific choice of monogenizer $\alpha \in \O_K$; when the choice of monogenizer $\alpha$ is implicit, we write just $K$ to denote the pair $(K,\alpha)$. Two monogenized \'{e}tale algebras $(K,\alpha)$ and $(K',\alpha')$ are said to be {\bf isomorphic} if there exists a ring isomorphism from $\O_K$ to $\O_{K'}$ taking $\alpha$ to $\pm \alpha' +m$ for some $m \in R$. 

For a binary $n$-ic form $f$ over $R$, let $R_f$ denote the ring of global sections of the closed subscheme of $\mathbb{P}^1_R$ cut out by $f$, and let $K_f = R_f \otimes_{R} F$ be the associated algebra of fractions. When $f$ has discriminant $\Delta(f) \neq 0$, we say that $f$ is {\bf maximal} if $R_f$ is the maximal order in $K_f$. We say that $f$ is {\bf monic} if its leading coefficient is $1$, i.e., if $f(1,0) = 1$, and we say that two monic forms $f$ and $f'$ are {\bf equivalent} if $f(x,y) = f'(x+my,y)$ for some $m \in R$. Observe that when $R = \Z$, the map sending $f \mapsto K_f$ defines a bijection from the set of monic maximal irreducible binary $n$-ic forms over $\Z$ to the set of monogenized degree-$n$ number fields, and further notice that this map remains a bijection upon passing to equivalence classes of forms and isomorphism classes of monogenized fields. Given a monogenized field $(K,\alpha)$, we may therefore speak of its {\bf defining binary form}, which is the homogenization of the minimal polynomial of $\alpha$.

\subsubsection{The height function} \label{sec-height}
We now restrict our consideration to monogenized \emph{cubic} fields. As explained in~\cite[\S1.1]{BSHpreprint}, it is natural to order isomorphism classes of such fields by the heights of their defining binary forms. 
For a monogenized cubic field $(K,\alpha)$ with defining binary form $f(x,y) = x^3 + ax^2y + bxy^2 + cy^3$, we define its {\bf height} as follows:
\begin{align*}
H(K,\alpha) \defeq H(f) \defeq \max\{|I(f)|^3, J(f)^2/4\}, \quad \text{where} \\
I(f) \defeq a^2 - 3b \quad \text{and} \quad J(f) \defeq -2a^3 + 9ab - 27c.
\end{align*}
Observe that $I(f(x,y)) = I(f(x+my,y))$ and $J(f(x,y)) = J(f(x+my,y))$ for all $m \in \Z$; in fact, the pair $(I(f),J(f))$ uniquely determines the equivalence class of the monic binary cubic form $f$. Thus, the height function $H$ is isomorphism-invariant and gives a well-defined notion of height for isomorphism classes of monogenized cubic fields.

\subsection{Main result}

In~\cite{BSHpreprint}, Bhargava, Hanke, and Shankar proved the following theorem, which determines the \emph{average} size of the $2$-class group $\on{Cl}_2(K) \defeq \on{Cl}(K)[2]$ and narrow $2$-class group $\on{Cl}_2^+(K) \defeq \on{Cl}^+(K)[2]$ as $K$ ranges through monogenized cubic fields ordered by height:
\begin{theorem}[\protect{\cite[Theorem~7]{BSHpreprint}}] \label{thm-prev}
Let $K$ run through all isomorphism classes of monogenized cubic fields with normal closure having Galois group isomorphic to $S_3$, ordered by height. Then:
 \begin{itemize}
  \item[{\rm (a)}] The average size of $\Cl_2(K)$ over totally real
    cubic fields $K$ is equal to $3/2$;
  \item[{\rm (b)}] The average size of $\Cl_2(K)$ over complex cubic
    fields $K$ is equal to $2$; and
  \item[{\rm (c)}] The average size of $\on{Cl}_2^+(K)$ over totally real cubic fields $K$ is equal to $5/2$.
  \end{itemize}
\end{theorem}

In this paper, we prove the following result, which is an analogue of Theorem~\ref{thm-prev} for the \emph{second moment} of the size of the $2$-class group and narrow $2$-class group of monogenized cubic fields:

\begin{theorem} \label{thm-main}
Let $K$ run through all isomorphism classes of monogenized cubic fields with normal closure having Galois group isomorphic to $S_3$, ordered by height. Then:
 \begin{itemize}
  \item[{\rm (a)}] The second moment of the size of $\Cl_2(K)$ over totally real
    cubic fields $K$ is at most $3$;
  \item[{\rm (b)}] The second moment of the size of $\Cl_2(K)$ over complex cubic
    fields $K$ is at most $6$; and
  \item[{\rm (c)}] The second moment of the size of $\on{Cl}_2^+(K)$ over totally real cubic fields $K$ is at most $9$.
  \end{itemize}
  Moreover, conditional on a tail estimate, the upper bounds of $3$, $6$, and $9$ are the exact values of the second moments.
\end{theorem}

In fact, we prove that Theorem~\ref{thm-main} holds---i.e., the second moment values remain unchanged---if we instead average over just those monogenized cubic fields satisfying any given set of local splitting conditions at finitely many primes, or even certain families of local splitting conditions at infinitely many primes; a similar generalization of Theorem~\ref{thm-prev} is stated in~\cite[Theorem~7]{BSHpreprint}.

\subsection{Comparison to class group heuristics}

  In their seminal paper~\cite{MR756082}, Cohen and Lenstra formulated definitive heuristics for the distribution of the class groups of quadratic fields; their heuristics were later generalized by Cohen and Martinet~\cite{MR866103} to predict the distribution of class groups of number fields of any fixed degree over a fixed base field. However, in the case of $2$-class groups of cubic number fields, the Cohen--Lenstra and Cohen--Martinet heuristics need to be adjusted to account for the fact that the base field contains second roots of unity. This observation was first made by Malle~\cite{MR2441080,MR2778658}, who made a concrete prediction for the moments of the $2$-class group distribution for cubic fields.\footnote{Cf.~the related work of Sawin--Wood~\cite{2301.00791}, which treats the case of $p$-class groups of $\Gamma$-extensions where the base field contains $p^{\mathrm{th}}$-roots of unity but $p \nmid \#\Gamma$. The Sawin--Wood heuristics are expected to generalize to the setting where $p \mid \#\Gamma$, and in the particular case where $p = 2$, $\Gamma = S_3$, and the base field is $\Q$, such generalized heuristics should agree with Malle's predictions.} In particular, the works of Malle makes the following predictions regarding the first and second moments of the $2$-class group $\on{Cl}_2(K)$ for cubic fields $K$:

\begin{conjecture}[Malle] \label{conj-theone}
Let $K$ run through all isomorphism classes of cubic fields ordered by discriminant. Then:
\begin{itemize}
    \item[\rm(a)] The average size of $\on{Cl}_2(K)$ over totally real cubic fields $K$ is equal to $5/4$, and the second moment of the size of $\on{Cl}_2(K)$ is equal to $15/8$; and
    \item[\rm(b)] The average size of $\on{Cl}_2(K)$ over complex cubic fields $K$ is equal to $3/2$, and the second moment of the size of $\on{Cl}_2(K)$ is equal to $3$.
\end{itemize}
\end{conjecture}

The claims in Conjecture~\ref{conj-theone}(a) and (b) about the average size of $\on{Cl}_2(K)$ were proven by Bhargava in~\cite[Theorem~5]{MR2183288}. Conjectured sizes for the first moment of the narrow $2$-class groups in $S_n$-fields are given in \cite{MR3873132}, but to the best of our knowledge, conjectures for the distribution (or even the second moment) are yet to be formulated. It is proven by Bhargava and Varma in~\cite[Theorem~1]{MR3369305} that the average size of $\on{Cl}_2^+(K)$ over totally real cubic fields $K$ is equal to $2$, in agreement with the conjectures of \cite{MR3873132}.

Theorem~\ref{thm-prev} shows that imposing the condition of monogenicity causes the nontrivial part of the $2$-class group to double on average, and it was the first result of its kind to demonstrate that imposing a global condition on a family of number fields can have an altering effect on the distribution of its class groups.  Indeed, as explained in~\cite[\S1]{BSHpreprint} and~\cite[\S1.5]{swacl}, there is a considerable body of computational and theoretical evidence to suggest that the class group distributions predicted by the heuristics of Cohen--Lenstra and others remain robust under the imposition of all manner of local and global conditions.  If the claims concerning second moments in Conjecture~\ref{conj-theone} are true, then our Theorem~\ref{thm-main} shows that this altering effect is not confined to the first moment of the $2$-class group distribution.

    In light of Theorems~\ref{thm-prev} and~\ref{thm-main}, it is natural to seek to formulate an analogue of Conjecture~\ref{conj-theone} for $2$-class groups of monogenized cubic fields. This is achieved in a forthcoming paper of Siad and Venkatesh~\cite{SiadVenkatesh}, who develop a set of heuristics that predict the entire distribution of the $2$-class groups of monogenized cubic fields (and indeed, monogenized fields of any fixed degree \mbox{at least $3$).} To do this, they show that the deviation in the average size of the $2$-class groups of monogenized cubic fields is explained through the notion of a ``spin structure'' on class groups, which is an arithmetic analogue of various notions of spin structure in geometry and topology. Roughly speaking, $2$-class groups of monogenized cubic fields possess a canonical spin structure, so the ``correct'' notion of automorphisms for such class groups ought to preserve the spin structure. In this respect, $2$-class groups of monogenized cubic fields should have fewer automorphisms than general cubic fields; because the Cohen--Lenstra philosophy implies that algebraic objects should occur with frequency inversely proportional to the sizes of their automorphism groups, we should expect $2$-class groups of monogenized cubic fields to be larger on average. In particular, the work of Siad--Venkatesh predicts second moments of $3$ (resp., $6$) for $2$-class groups of totally real (resp., complex) cubic fields, predictions that agree with our Theorem~\ref{thm-main}(a) and (b).

\subsection{The idea of the proof} \label{sec-idea}

In this subsection, we explain the idea of the proof of Theorem~\ref{thm-main}(a) and (b); the remaining part concerning the narrow $2$-class group is more subtle and requires a finer analysis, but the structure of the proof is nonetheless analogous.

Computing the second moment of $\#\on{Cl}_2(K)$ amounts to determining the average size of the set $\on{Cl}_2(K)^2$ of pairs $(\alpha_1, \alpha_2)$ of $2$-torsion ideal classes of (the ring of integers of) $K$, where $K$ ranges over monogenized cubic number fields. To see why the second moment of $\#\on{Cl}_2(K)$ should be $3$ in the totally real case and $6$ in the complex case, fix a $2$-torsion class $\alpha_1$ of such a field $K$, and consider the following heuristic argument:

\vspace*{0.2cm}
\noindent \emph{Case 1}: 
Suppose $\alpha_1$ corresponds to the trivial element of $\on{Cl}_2(K)$. Then $K$ can be arbitrary, and~\cite[part (a) of Theorem~4]{BSHpreprint} implies that the average number of choices for $\alpha_2$ on average is $\frac{3}{2}$ in the totally real case and $2$ in the complex case.

\vspace*{0.2cm}
\noindent \emph{Case 2}: 
Suppose $\alpha_1$ corresponds to a non-trivial element of $\on{Cl}_2(K)$. Then~\cite[part (b) of Theorem~4]{BSHpreprint} implies that the average number of choices for $\alpha_1$ is $\frac{1}{2}$ in the totally real case and $1$ in the complex case. Moreover, the monogenized cubic field $K$ cannot be arbitrary---its class group has a marked non-trivial $2$-torsion element, namely the one corresponding to $\alpha_1$. We expect that this marked element gives an extra generator for the $2$-class group $100\%$ of the time, and so, the average number of choices for $\pi_2$ should be $2 \times \frac{3}{2} = 3$ in the totally real case and $2 \times 2 = 4$ in the complex case, twice as many choices as in case 1.

\vspace*{0.2cm}
\noindent Combining the results of cases $1$ and $2$, we conclude that the average number of choices for the pair $(\alpha_1, \alpha_2)$ should be $1 \times \frac{3}{2} + \frac{1}{2} \times 3 = 3$ in the totally real case and $1 \times 2 + 1 \times 4 = 6$ in the complex case, as desired.

In the related work~\cite{BSSpreprint}, we prove that the second moment of the size of the $2$-Selmer group of elliptic curves is at most $15$ by making rigorous an argument similar to the one detailed above. In~\cite{BSSpreprint}, the bound of $15$ arises from evaluating a certain formula involving a product of local masses over all places $v$ of $\Q$. In this paper, we show in \S\ref{sec-selstructs} that the methods of~\cite{BSSpreprint} can be extended to obtain second moments not just for the usual notion of $2$-Selmer group, but in fact for other notions of $2$-Selmer group arising from the theory of Selmer structures on elliptic curves. Then, in \S\ref{sec-unramstruct}, we realize the $2$-class group of a monogenized cubic field $K_f$ as the $2$-Selmer group corresponding to the ``unramified'' Selmer structure for the elliptic curve $y^2 = f(x)$, and we develop a similar characterization for the narrow $2$-class group. We finish by demonstrating how the local mass calculations in~\cite{BSSpreprint} can be adjusted to obtain the second moment for the $2$-Selmer groups corresponding to ordinary and narrow $2$-class groups.


\begin{remark}
    The proof idea articulated above suggests a na\"{i}ve heuristic for class group distributions, by means of which one can determine the second moment, and indeed every higher moment, knowing only the first moment. The distribution of $p$-Selmer groups of elliptic curves conjectured by Poonen-Rains~\cite{MR2833483} obeys this heuristic for every prime $p$, as does the distribution of $2$-torsion in class groups of cubic fields predicted by Malle~\cite{MR2441080,MR2778658}. But this na\"{i}ve heuristic does not hold for most class group distributions; e.g., the moments for the distribution of $p$-torsion in class groups of quadratic fields, as predicted by Cohen--Lenstra~\cite{MR756082}, do not agree with the moments predicted by the na\"{i}ve heuristic.
\end{remark}

\section{Selmer structures for families of elliptic curves} \label{sec-selstructs}

In this section, we set the notation for elliptic curves associated to monic binary cubic forms (\S\ref{sec-ell}) and define Selmer structures on such curves (\S\ref{sec-defstructs}). Then we explain how the method of~\cite{BSSpreprint} for obtaining the second moment of the size of the $2$-Selmer groups of elliptic curves---as well as the methods of~\cite{MR3272925,MR3275847,4sel,5sel} for determining the average size of the $n$-Selmer groups of elliptic curves for $n \in \{2,3,4,5\}$---can be extended to prove analogous results (Theorem~\ref{thm-countingselmers}) for Selmer groups corresponding to Selmer structures (\S\ref{sec-results}).

\subsection{Elliptic curves associated to monic binary cubic forms} \label{sec-ell}

Consider the family of elliptic curves over $\Q$ in the short Weierstrass form $$E^{I,J} \colon y^2 = x^3 -\tfrac{I}{3}x-\tfrac{J}{27},$$
where $I,J$ are integers satisfying the following four conditions: $3^4 \mid I$, $3^6 \mid J$, $3^9 \nmid I$ if $3^{15} \mid J$, and $p^4 \nmid I$ if $p^6 \mid J$ for every prime $p > 3$. These four conditions are imposed to ensure that each elliptic curve over $\Q$ admits a unique model of the form $E^{I,J}$. For such models, the \textbf{discriminant} is given explicitly by
$$\Delta\big(E^{I,J}\big) \defeq \Delta(I,J) \defeq (4I^3 - J^2)/27.$$
Given a maximal monic binary cubic form $f \in \Z[x,y]$ with $x^2y$-coefficient equal to $a$, we define the {\bf elliptic curve associated to $f$} by
$$E_f \defeq E^{3^4I(f),3^6J(f)}, 
$$
where $I(f)$ and $J(f)$ are defined as in \S\ref{sec-height}. Note that the curve $E_f$ is independent of the choice of representative of the equivalence class of $f$. Moreover, observe that the binary cubic form $x^3 - I(f)x/3 - J(f)/27$ is equivalent to $f$ over $\Q$; the extra factors of $3^4$ and $3^6$ are included in the definition of $E_f$ to ensure that the equation for $E_f$ has integral coefficients.

\subsection{Defining Selmer structures} \label{sec-defstructs}

Let $E$ be an elliptic curve over a field $F$. For any positive integer $n$, we have the exact sequence
\begin{equation*}
0\to E[n]\to E\to E \to 0
\end{equation*}
of $G_F$-modules, where $G_F$ denotes the absolute Galois group of $F$. This yields the long exact sequence
\begin{equation*}
0\to E(F)/nE(F)\to H^1(F,E[n])\to H^1(F,E)[n]\to 0.
\end{equation*}
The map from $E(F)$ to $H^1(F,E[n])$ is called the {\bf Kummer map}.
We will use the above exact sequence in the cases $F=\Q$ and $F=\Q_v$ for places $v$ of $\Q$, and we denote the Kummer map in these cases by $\delta$ and $\delta_v$, respectively. 

For every place $v$ of $\Q$, we have a natural injection $G_{\Q_v}\hookrightarrow G_\Q$. This yields a group homomorphism $H^1(\Q,T)\to H^1(\Q_v,T)$ for any $G_\Q$-module $T$. Let $E$ be an elliptic curve over $\Q$, and let $n$ be a positive integer. The $n$-Selmer group is then defined to be the set of classes in $H^1(\Q,E[n])$ that are in the image of the Kummer map everywhere locally:
\begin{equation}\label{eqnsel}
\Sel_n(E)\defeq \ker\Bigl(H^1({\Q},E[n])\to\prod_v H^1({\Q_v},E[n])/\Im(\delta_v)\Bigr).
\end{equation}
More generally, we may impose different local conditions at each place $v$. An {\bf $n$-Selmer structure} $\cS(E)$ for an elliptic curve $E$ over $\Q$ is a collection $\cS(E) = (\cS_v(E))_v$ of \emph{subgroups} $\cS_v(E) \subset H^1(\Q_v,E[n])$, such that $\cS_v(E)$ is the image of the Kummer map $\delta_v$ for all but finitely many places $v$ (cf.~\cite{MR2031496} and~\cite{MR1638477} for more on Selmer structures). We then define the Selmer group associated to $\cS \defeq \cS(E)$ via
\begin{equation}\label{eqFsel}
\Sel_{\cS}(E)\defeq \ker\Bigl(H^1({\Q},E[n])\to\prod_v H^1({\Q_v},E[n])/\cS_v(E)\Bigr).
\end{equation}
Since $\Sel_n(E)$ is a finite group, it follows that $\Sel_{\cS}(E)$ is also finite.
\medskip

For every prime $p$, let $\Lambda_p\subset\Z_p^2\smallsetminus\{\Delta=0\}$ be a closed subset whose boundary has measure $0$. We associate the set $\EE=\EE_\Lambda$ of elliptic curves to the collection $\Lambda=(\Lambda_p)_p$, where $E^{I,J}\in\EE$ if and only if $(I,J) \in\Lambda_p$ for all primes $p$; here, we restrict our choice of the sets $\Lambda_p$ so that the corresponding pairs $(I,J)$ satisfy the four divisibility conditions described in \S\ref{sec-ell}. Following \cite[\S3]{MR3272925}, we say that $\EE$ is {\bf defined by congruence conditions}. We may also impose ``conditions at infinity'' on $\EE$ by fixing a compact subset $\Lambda_\infty\subset \{(I,J):\Delta(I,J)\neq 0\}$ whose boundary has measure $0$, and saying that $E^{I,J} \in \EE$ if $(t^2I,t^3J)$ belongs to $\Lambda_\infty$ for some $t > 0$. In practice, we will take $\Lambda_\infty$ to be the set of elements of height at most $1$ in one of the following three sets: $\{(I,J):\Delta(I,J)>0\}$, $\{(I,J):\Delta(I,J)<0\}$, or $\{(I,J):\Delta(I,J)\neq 0\}$. The collection $\EE_\Lambda$ is said to be a {\bf large family} if for all but finitely many primes $p$, the set $\Lambda_p$ contains every pair $(I,J)$ such that $p^2\nmid\Delta(I,J)$.

Our next goal is to define $n$-Selmer structures uniformly for elliptic curves in large families. For an elliptic curve $E$ over $\Q$, we think of a subgroup $\cS_v\subset H^1(\Q_v,E[n])$ as being a {\bf local condition} for $v$. We then define the following five local conditions:
\begin{itemize}
\item[{\rm (a)}] The {\bf relaxed local condition} at $v$ corresponds to $\cS_v=H^1(\Q_v,E[n])$.
\item[{\rm (b)}] The {\bf soluble local condition} at $v$ corresponds to $\cS_v=\Im(\delta)$.
\item[{\rm (c)}] The {\bf unramified local condition} at $v$ corresponds to $\cS_v=H^1_\ur(\Q_v,E[n])$, the subgroup of {\bf unramified cohomology classes} defined by $$H^1_\ur(\Q_v,E[n])\defeq \ker\big(H^1(\Q_v,E[n])\to H^1(I_v,E[n](\ol{\Q_v}))\big),$$ where $I_v\subset G_{\Q_v}$ is the inertia subgroup.
\item[{\rm (d)}] The {\bf soluble and unramified local condition} at $v$ corresponds to $\cS_v=\on{im}(\delta_v)\cap H^1_\ur(\Q_v,E[n])$.
\item[{\rm (e)}] The {\bf strict local condition} at $v$ corresponds to $\cS_v=0$.
\end{itemize}
An {\bf $n$-Selmer structure $\cS$ for a family of elliptic curves $\EE$} gives, for each place $v$ of $\Q$, one of the above five local conditions, namely relaxed, soluble, unramified, soluble and unramified, or strict. It is then clear that for any elliptic curve $E$ over $\Q$, we obtain an $n$-Selmer structure also denoted $\cS(E)=(\cS_v(E))_v$ by taking $\cS_v(E)$ to be the corresponding subgroup of $H^1(\Q_v,E[n])$ described above. Such a structure $\cS$ for $\EE$ is said to be {\bf sub-soluble} if, for every place $v$ and every elliptic curve $E\in\EE$, the given local condition $\cS_v(E)$ is a subset of the soluble local condition on $E$. That it, $\cS_v(E)$ must be contained within the image $\on{im}(\delta_v)$ of the Kummer map in $H^1(\Q_v,E[n])$.
We say that $\cS$ is {\bf acceptable} if for all but finitely many places $v$, $\cS_v$ is the soluble local condition or the unramified local condition.

\subsection{Results on moments of Selmer groups corresponding to Selmer structures} \label{sec-results}

In this section, we state and prove our main result, Theorem~\ref{thm-countingselmers}, for obtaining moments of Selmer groups associated to Selmer structures on families of elliptic curves. The moments are expressed in terms of local masses, which we now introduce. For a place $v$ of $\Q$, define the {\bf local mass} at $v$ associated to an $n$-Selmer structure $\cS$ for a large family $\EE=\EE_\Lambda$ of elliptic curves over $\Q$ as follows:
\begin{equation}\label{eqMass}
\cM_v(n;\EE;\cS)\defeq
\displaystyle\frac{\displaystyle\int_{(I,J)\in\Lambda_v}\displaystyle\frac{\#\cS_v(E^{I,J})}{\#E^{I,J}(\Q_v)[n]}dIdJ}
{\displaystyle\int_{(I,J)\in\Lambda_v}dIdJ},
\end{equation}
where $dIdJ$ denotes the Haar measure on $\Z_v^2$, normalized for each finite place $v = p$ so that $\on{Vol}(\Z_p^2) = 1$ and for $v = \infty$ so that $\on{Vol}(\Z^2 \backslash \R^2) = 1$. Then define the total mass as the product of local masses:
\begin{equation*}
\cM(n;\EE;\cS) \defeq \prod_v \cM_v(n;\EE;\cS).
\end{equation*}
For a real number $X>0$, let $\EE(X)\defeq \{E^{I,J}\in\EE : (X^{-1/3}I,X^{-1/2}J)\in\Lambda_\infty\}$.
We define the $k^{\mathrm{th}}$ moment of the size of $\Sel_{\cS(E)}(E)$ to be
\begin{equation*}
\Avg^{(k)}(n;\EE;\cS) \defeq \lim_{X\to\infty}
\frac{\displaystyle\sum_{E\in\EE(X)}
\#\Sel_{\cS(E)}(E)^k}{\#\EE(X)},
\end{equation*}
if the limit exists. When $k=1$, we suppress the superscript in $\Avg^{(k)}$.

Suppose $d$ is a divisor of $n$. If $\cS$ is an $n$-Selmer structure on a family of elliptic curves, then $\cS$ naturally induces a $d$-Selmer structure, denoted $\cS^{(d)}$, on this family as well. With this notation, the subgroup of $\Sel_{\cS(E)}(E)$ consisting of elements having order dividing $d$ is isomorphic to $\Sel_{\cS^{(d)}(E)}(E)$. Then we have the following result.
\begin{theorem} \label{thm-countingselmers}
Let $\EE=\EE_\Lambda$ be a large family of elliptic curves, let $n \in \{2,3,4,5\}$, and let $\cS$ be an acceptable and sub-soluble $n$-Selmer structure for $\EE$. Then we have
\begin{equation*}
\begin{array}{rcl}    
\Avg(n;\EE;\cS)&=&\displaystyle 1+n\cM(n;\EE;\cS),
\\[.1in]
\Avg(4;\EE;\cS)&=&\displaystyle 1+2\cM(2;\EE;\cS^{(2)})+4\cM(4;\EE;\cS) \quad \text{when $n = 4$},
\\[.1in]
\Avg^{(2)}(2;\EE;\cS)&\leq&\displaystyle 1+6\cM(2;\EE;\cS)+8\cM(2;\EE;\cS)^2 \quad\,\,\,\, \text{when $n = 2$},
\end{array}
\end{equation*}
where in the last line of the display, we have equality conditional on a tail estimate.
\end{theorem}
\begin{remark}
    Although our method of proof yields an upper bound for $\Avg^{(2)}(2;\EE;\cS)$, it is sufficiently robust to obtain the following generalization: if $\cS$ and $\wt{\cS}$ are sub-soluble $2$-Selmer structures for $\EE$ such that $\cS$ is a sub-structure of $\widetilde{\cS}$, then we have
    \begin{equation} \label{eq-slightgen}
        \Avg^{(2)}(2;\EE;\widetilde{\cS}) - \Avg^{(2)}(2;\EE;\cS) \leq 6\big(\cM(2;\EE;\widetilde{\cS}) - \cM(2;\EE;\cS)\big)+8\big(\cM(2;\EE;\widetilde{\cS})^2 - \cM(2;\EE;\cS)^2\big)
    \end{equation}
    Such generalizations are of course immediate for the other moments considered in Theorem~\ref{thm-countingselmers}.
\end{remark}
\begin{proof}
For $n \in \{2,3,4,5\}$, we claim that the average number of elements in $\Sel_{\cS(E)}(E)$ that have exact order $n$, is equal to $n\cM(n;\EE;\cS)$. 
To prove this claim, we follow the setup and methods from \cite{MR3272925,MR3275847,4sel,5sel}. 
We consider semisimple algebraic groups $G_n$ and coregular representations $V_n$ of $G_n$, such that the rings of invariants for the actions of $G_n$ on $V_n$ are freely generated by two elements, denoted $I$ and $J$. 
The first step is to apply results of Cremona--Fisher--Stoll \cite{MR2728489} and Fisher \cite{MR3101076} to obtain, for an elliptic curve $E^{I,J}$ over $\Q$, a bijection
\begin{equation*}
\{\sigma\in\Sel_n(E^{I,J}):\ord(\sigma)=n\}\to G_n(\Q)\backslash V_n(\Z)^{(I,J),\irr,\ls},
\end{equation*}
where the right-hand side is the set of $G_n(\Q)$-equivalence classes on the set of elements $v\in V_n(\Z)$ with invariants\footnote{When $n=2$, in \cite{MR3272925}, we actually choose the invariants to be $2^4I$ and $2^6J$. However, by simply rescaling those polynomial invariants, we can cover $n\in\{2,3,4,5\}$ in a uniform narrative.} $I(v)=I$ and $J(v)=J$
that are {\bf irreducible} and {\bf everywhere locally soluble}. The condition of irreducibility is a global one, ensuring that $w\in V_n(\Z)$ corresponds to an element in $\sigma(w)\in H^1(\Q,E^{I,J}[n])$ which has exact order $n$. The condition of being everywhere locally soluble is a local one, which ensures that for all places $v$ of $\Q$, the image of $\sigma(w)$ in $H^1({\Q_v},E^{I,J}[n])$ lies in the image of the Kummer map. For our situation, we note that since $\cS$ is assumed to be sub-soluble, we have $\Sel_{\cS(E)}(E)\subset\Sel_n(E)$ for all elliptic curves $E$. Therefore, we obtain the bijection
\begin{equation*}
\{\sigma\in\Sel_{\cS(E)}(E):\ord(\sigma)=n\}\to G_n(\Q)\backslash V_n(\Z)^{(I,J),\irr,\cS},
\end{equation*}
where $E=E^{I,J}$, and where the superscript $\cS$ is used to prescribe the condition that $w\in V_n^{I,J}(\Z)$ is such that the image of $\sigma(w)$ in 
$H^1({\Q_v},E[n])$ belongs to $\cS_v(E)$ for all places $v$ of $\Q$.

For a real number $X>0$, let $X.\Lambda_\infty \defeq \{(I,J)\in\R^2 : (X^{-1/3}I,X^{-1/2}J)\in\Lambda_\infty\}$.
The second step in \cite{MR3272925,MR3275847,4sel,5sel} is to apply geometry-of-numbers methods to prove that we have the asymptotic
\begin{equation*}
\#\{w\in G_n(\Z)\backslash V_n(\Z)^\irr:(I(w),J(w))\in X.\Lambda_\infty\}\sim
|\J_n|\cM_\infty(n;\EE,\cS)\Vol(G_n(\Z)\backslash G_n(\R))\Vol(X.\Lambda_\infty),
\end{equation*}
where $\J_n$ is some nonzero rational constant.
This is proven in \cite{MR3272925,MR3275847,4sel,5sel} when $\Lambda_\infty$ consists of the set of elements in $\R^2$ with height less than $1$, but the case for general $\Lambda_\infty$ follows identically.

The final step is to impose local conditions (at all primes $p$) on the points in $V_n(\Z)$ being counted such that: (a) the count of $G_n(\Z)$-orbits is replaced with the count of $G_n(\Q)$-equivalence classes; and (b) $w\in V_n(\Z)$ is only counted when $w\in V_n(\Z)^{\cS}$. That asymptotics can be obtained even when infinitely many local conditions are imposed on $V_n(\Z)$ requires an input of a tail estimate. Such tail estimates are proved in \cite{MR3272925,MR3275847,4sel,5sel} and apply identically in our situation since we assume that $\cS$ is acceptable. With this input, we obtain the asymptotic
\begin{equation*}
\frac{\#\{w\in G_n(\Z)\backslash V_n(\Z)^{\irr,\cS}:(I(w),J(w))\in X.\Lambda_\infty\}}{\Vol(X.\Lambda_\infty)}\sim
|\J_n|\cM_\infty(n;\EE,\cS)\Vol(G_n(\Z)\backslash G_n(\R))\prod_p\nu_p,
\end{equation*}
where $\nu_p$ is the $p$-adic density of the local condition imposed at $p$. When $\cS_p$ is the soluble condition, $\nu_p$ is computed in \cite{MR3272925,MR3275847,4sel,5sel} to be
\begin{equation*}
\nu_p=|\J_n|_p\Vol(G_n(\Z_p))\int_{(I,J)\in\Lambda_v}\displaystyle\frac{\#E^{I,J}(\Q_p)/2E^{I,J}(\Q_p)}{\#E^{I,J}(\Q_p)[n]}dIdJ,
\end{equation*}
where $\#E^{I,J}(\Q_p)/2E^{I,J}(\Q_p)$ is used to parametrize the image of the Kummer map. In our situation, an identical computation yields
\begin{equation*}
\nu_p=|\J_n|_p\int_{(I,J)\in\Lambda_v}\displaystyle\frac{\#\cS_v(E^{I,J})}{\#E^{I,J}(\Q_p)[n]}dIdJ=
|\J_n|_p\Vol(G_n(\Z_p))\cM_{p}(n;\EE;\cS)\Vol(\Lambda_p).
\end{equation*}
The number of elements $E^{I,J}\in\EE$ with $(I,J)\in X.\Lambda_\infty$ is asymptotic to $\Vol(X.\Lambda_\infty)\prod_p\Vol(\Lambda_p)$.
Putting everything together, we see that the average number of elements in $\Sel_{\cS(E)}(E)$ having exact order $n$, over $E\in\EE$, is equal to
\begin{equation*}
\lim_{X \to \infty} \frac{\#\{w\in G_n(\Z)\backslash V_n(\Z)^{\irr,\cS}:(I(w),J(w))\in X.\Lambda_\infty\}}{\Vol(X.\Lambda_\infty)\prod_p\Vol(\Lambda_p)}
= \tau(G_n)\cM(n;\EE;\cS),
\end{equation*}
where $\tau(G_n)=\Vol(G_n(\Z)\backslash G_n(\R))\prod_p\Vol(G_n(\Z_p))$ is the Tamagawa number of $G_n$. Since this Tamagawa number is $n$, the first two claims of the theorem follow.

\medskip

To prove the third claim, we follow the setup and methods of \cite{BSSpreprint}, where the strategy sketched in \S\ref{sec-idea} is made rigorous with the $2$-class groups of monogenized cubic fields replaced by the $2$-Selmer groups of elliptic curves. In this situation, we consider a pair $(\sigma_1,\sigma_2)\in\Sel_{\cS(E)}(E)^2\subset\Sel_{2}(E)^2$, where $E=E^{I,J}$ is an elliptic curve over $\Q$. If $\sigma_1$ is trivial, then \mbox{$\#(\sigma_1\times\Sel_{\cS(E)}(E))=\#\Sel_{\cS(E)}(E)$,} and the average value of this size over $E\in\EE$ is $1+2\cM(2;\EE;\cS)$.

Now suppose that $\sigma_1$ is not the identity element. This time, we map $(\sigma_1,\sigma_2)$ to the pair $(f,\sigma_2)$, where $f\in V_2(\Z)^{(I,J),\irr,\cS}$, and $\sigma_2\in\Sel_{\cS(E)}(E)$. Here, $f$ is a binary quartic form with invariants $I$ and $J$, having no linear factor over $\Q$, and satisfying certain congruence conditions imposed by $\cS$. In particular, these congruence conditions imply that $f$ is locally soluble (since $\cS$ was assumed to be sub-soluble). Moreover, $E=\Jac_f$ is determined by $f$ as being the Jacobian of the genus-$1$ curve $C_f$ cut out by the equation $z^2=f(x,y)$ in weighted projective space. Applying parametrization results of \cite{swacl}, it is proved in \cite{BSSpreprint} that nontrivial elements in $\Sel_2(C_f)$ maps injectively into certain {\bf special} and {\bf irreducible} $\PSO_{\cA}(\Z)$-orbits on $B\in W_4(\Z)$, the space of integral quaternary quadratic forms, such that $\det(x\cA+yB)=f^\mon(x,y)$. Here $\cA$ is the split quaternary quadratic form of determinant $1$, and $f^\mon(x,y)$ denotes the {\bf monicization} of $f(x,y)$. Furthermore, for each $f$ arising from a nontrivial $2$-Selmer element of $E$, there are two ``trivial'' elements of $\Sel_2(C_f)$ whose corresponding integral orbit is not irreducible.
The definitions of being special, irreducible, and monicizations are not necessary for us. Suffice it to say that in \cite{BSSpreprint}, we determine asymptotics for the number of special and irreducible orbits $B$, satisfying specified local conditions, such that $\det(x\cA +yB)$ is the monicization of a binary quartic form having bounded invariants.

Thus, as before, the only difference between our case and \cite{BSSpreprint} is that the local conditions imposed by $\cS$ differ from those imposed by local solubility, and this will have the effect of changing the local masses. In \cite{BSSpreprint}, the average value of $\#\Sel_2(C_f)-2$ (the $-2$ is the contribution from the trivial elements) over binary quartic forms $f$ corresponding to nontrivial elements in $\Sel_2(E)$, over $E\in\EE$ is computed to be less than or equal to
\begin{equation*}
    \tau(\PGL_2)\tau(\PSO_\cA)\left(\prod_v
    \int_{(I,J)\in\Lambda_v}\displaystyle\frac{\#E^{I,J}(\Q_p)/2E^{I,J}(\Q_p)}{\#E^{I,J}(\Q_p)[n]}dIdJ
    \right)^2,
\end{equation*}
where once again, $E^{I,J}(\Q_p)/2E^{I,J}(\Q_p)$ is used to parametrize the image of the Kummer map, and the inequality is in fact an equality conditional on a tail estimate. In our situation, with $\Sel_2$ replaced by $\Sel_\cS$, we obtain an average of $8\cM(2;\EE;\cS)^2$.

Combining all of the different contributions to $\Avg^{(2)}(2;\EE;\cS)$, we obtain a summand of $1+2\cM(2;\EE;\cS)$ from pairs $(\sigma_1,\sigma_2)\in\Sel_{\cS(E)}(E)$ with $\sigma_1$ being trivial, a summand of $4\cM(2;\EE;\cS)$ from pairs $(\sigma_1,\sigma_2)$ where $\sigma_1$ (corresponding to $f$) is nontrivial and $\sigma_2$ is one of the two trivial elements in $\Sel_{\cS(\on{Jac}_f)}(\on{Jac}_f)$, and a summand of
$8\cM(2;\EE;\cS)^2$ from pairs $(\sigma_1,\sigma_2)$ where both $\sigma_1$ and $\sigma_2$ are nontrivial. Adding them up, we obtain the result.
\end{proof}

\section{A Selmer structure for $2$-torsion in class groups of monogenized cubic fields} \label{sec-unramstruct}

In this section, we first demonstrate that the $2$-class group of a monogenized cubic field can be interpreted as the Selmer group associated to the unramified Selmer structure on the corresponding elliptic curve (\S\ref{sec-ordinary}). Then we explain how the narrow $2$-class group can be captured using a combination of two different Selmer structures in conjunction with the principle of inclusion-exclusion (\S\ref{sec-narrowgen}). This puts us in a position to apply Theorem~\ref{thm-countingselmers} to obtain the desired second moments, which we do for the ordinary $2$-class group in \S\ref{sec-proofab} and for the narrow $2$-class group in \S\ref{sec-proofc}.

\subsection{The case of the ordinary $2$-class group} \label{sec-ordinary}

We begin with the following well known lemma regarding the Galois cohomology of the $2$-torsion subgroup $E[2]$ for elliptic curves $E$.

\begin{lemma}\label{lemEC2torGC}
Let $F$ be a field of characteristic zero, let $f \in F[x,y]$ is a $($separable$)$ monic binary cubic form, let $E = E_f$, and let $K = K_f$ denote the \'etale algebra $F[x]/(f(x,1))$. Then we have an exact sequence
\begin{equation}\label{eqses}
1\to E[2]\to \Res_{K/F}\mu_2 \to \mu_2\to 1
\end{equation}
 of $G_F$-modules. Moreover we have the isomorphism
\begin{equation*}
H^1(F,E[2])\simeq (K^\times/K^{\times2})_{\on{N}\equiv1},
\end{equation*}
where the right-hand side consists of elements in $K^\times/K^{\times2}$ having square norm in $F^\times$.
Finally, if $H$ is a subgroup of $G_F$, then an element in $H^1(F,E[2])$, corresponding to the class of $\alpha\in K^\times$, maps to the identity in $H^1(H,E(F^{\on{sep}})[2])$ if and only if $H$ acts trivially on $\sqrt{\alpha}$.
\end{lemma}
\begin{proof}
Over $\overline{F}$, the algebraic closure of $F$, the polynomial $f(x)$ has three distinct roots, say $\alpha_1$, $\alpha_2$, and $\alpha_3$. On the level of $\overline{F}$-points, we have $E(\overline{F})[2]=\{\infty,(\alpha_1,0),(\alpha_2,0),(\alpha_3,0)\}$, with the $G_F$ action that fixes $\infty$ and acts on the other three points in the same way as $G_F$ acts on the triple $(\alpha_1,\alpha_2,\alpha_3)$. Furthermore, the $\overline{F}$-points of $\Res_{K/F}\mu_2$ can be written as triples $(\pm1,\pm1,\pm1)$, where the $i^{\mathrm{th}}$ coordinate corresponds to $\alpha_i$ and the $G_F$ acts on $\Res_{K/F}\mu_2$ by correspondingly permuting the coordinates. The map $\Res_{K/F}\mu_2 \to \mu_2$ is simply the norm map given by $(\theta_1,\theta_2,\theta_3)\mapsto (\theta_1\theta_2\theta_3)$.
The map $E[2]\to \Res_{K/F}\mu_2$ given by
\begin{equation*}
\infty\mapsto (1,1,1),\;\;(\alpha_1,0)\mapsto (1,-1,-1),\;\;(\alpha_2,0)\mapsto (-1,1,-1),\;\;(\alpha_3,0)\mapsto (-1,-1,1),
\end{equation*}
is an injective homomorphism respecting the action of $G_F$. Moreover, its image is exactly kernel of the norm map. This concludes the proof of the first assertion.

To see that the second assertion is true, we take the long exact sequence corresponding to \eqref{eqses} and obtain the exact sequence
\begin{equation*}
    1\to H^1(F,E[2])\to H^1(F,\Res_{K/F}\mu_2)\to H^1(F,\mu_2).
\end{equation*}
Kummer theory asserts that we have
\begin{equation*}
H^1(F,\Res_{K/F}\mu_2)\simeq K^\times/K^{\times2},\quad 
H^1(F,\mu_2)\simeq F^\times/F^{\times2},
\end{equation*}
concluding the proof of the second assertion of the lemma.

Making the Kummer map explicit, we see that an element $\alpha\in K^\times/K^{\times2}$ corresponds to the following cocycle $\sigma$ in $H^1(F,\Res_{K/F}\mu_2)$:
\begin{equation*}
\sigma\colon \gamma\mapsto \gamma(\beta)/\beta,
\end{equation*}
where $\beta\in K\otimes_F\overline{F}$ satisfies $\beta^2=\alpha$. It follows that $\sigma$ maps to the identity in $H^1(H,\Res_{K/F}\mu_2)$ if and only if $h(\beta)=\beta$ for all $\beta\in H$, or equivalently, if $H$ acts trivially on $\beta$.
\end{proof}

For an elliptic curve $E$ and a place $v$ of $\Q$, recall that a cohomology class in $H^1({\Q_v},E[n]$ is said to be unramified if it belongs to the kernel of the map $H^1({\Q_v},E[n])\to H^1(I_v,E(\ol{\Q_v})[n])$, where $I_v\subset G_{\Q_v}$ is the inertia subgroup. We define the subgroup $H_\ur^1({\Q},E[n])$ of {\bf everywhere unramified cohomology classes} to be the subgroup of elements that are unramified everywhere locally.
Let $(K,\alpha)$ be a monogenized cubic field, and let $f$ denote the monic binary cubic form such that $f(x,1)$ is the characteristic polynomial of $\alpha$. 
We next relate the $2$-torsion in the dual of the class group of $K$ to $H_\ur^1({\Q},E_f[2])$.

\begin{proposition} \label{prop-classgroup}
Let $K$, $f$, and $E_f$ be as above and let $\Cl^*(K)$ denote the group dual to $\Cl(K)$. Then we have a natural bijection
\begin{equation*}
\Cl^*(K)[2] \overset{\sim}\longrightarrow H_\ur^1({\Q},E_f[2]).
\end{equation*}
\end{proposition}
\begin{proof}
By class field theory, nontrivial elements in $\Cl^*(K)[2]$ correspond bijectively to unramified quadratic field extensions of $K$. Meanwhile, applying Lemma \ref{lemEC2torGC} to the elliptic curve $E_f$, we see that $H^1(\Q,E_f[2])$ is isomorphic to $(K^\times/K^{\times2})_{\on{N}\equiv1}$. Furthermore, an element $\sigma\in H^1(\Q,E_f[2])$, corresponding to the square class of $\alpha\in K^\times$ is unramified at a place $v$ of $\Q$ if and only if it maps to $0$ in $H^1(I_v,E_f[2](\ol{\Q_v}))$, where $I_v$ is the inertia subgroup of $G_{\Q_v}$. Applying Lemma \ref{lemEC2torGC} again, we see that this happens if and only if $I_v$ acts trivially on $\sqrt{\alpha}$. In other words, $\sigma$ is unramified if and only if the quadratic field $K(\sqrt{\alpha})$ is unramified over $K$. Moreover, if $K(\sqrt{\alpha})/K$ is unramified, then $\alpha$ must have square norm down to $\Q$. Therefore, nontrivial elements in $H^1_\ur(\Q,E_f[2])$ correspond bijectively to unramified quadratic extensions of $K$, yielding the result.
\end{proof}

We now consider the local picture. Let $v$ be a place of $\Q$ and let $(K,\alpha)$ be a monogenized cubic \'etale algebra over $\Q_v$ with corresponding monic binary cubic form $f$, and let $E = E_f$. The next lemma determines the size of $H^1_\ur({\Q_v},E[2])$.
\begin{lemma}\label{lemunramsize}
With notation as above, we have
\begin{equation*}
\#H^1_\ur(\Q_v,E[2])=\begin{cases}
\#E(\Q_v)[2]&\text{if $v$ is a finite place,}\\[.15in]
1&\text{if $v$ is an infinite place}.
\end{cases}
\end{equation*}
\end{lemma}
\begin{proof}
The lemma follows immediately from a case by case analysis of the possibilities of $K_v\defeq \Q_v[x]/(f(x,1))$. When $v=p$ is a finite place, $K_p$ can be $\Q_p^3$, $\Q_p\oplus\Q_{p^2}$, $\Q_{p^3}$, $\Q_p\oplus F_{p,2}$, or $F_{p,3}$, where $\Q_{p^\ell}$ denotes the (unique) unramified degree-$\ell$ extension of $\Q_p$, and where $F_{p,\ell}$ is a totally ramified degree-$\ell$ extension of $\Q_p$. Thus, $K_p$ is a sum of one, two, or three field extensions of $\Q_p$. The size of $E(\Q_p)[2]$ is equal to $1+\beta$, where $\beta$ is the number of $\Q_p$-rational roots of $f(x,1)$, which is in turn equal to $1$ plus the number of times $\Q_p$ occurs as a factor of $K_p$. Hence, as we go over the five possibilities $K_p$, we see that $\#E(\Q_v)[2]$ equals $4$, $2$, $1$, $2$, and $1$, respectively. Meanwhile, we have $\#H^1_\ur(\Q_v,E[2])$ is equal to the number of elements $\theta\in K_p^\times/K_p^{\times^2}$ that have square norm such that $K_p(\sqrt{\theta})$ is unramified. This number is again easily seen to equal $4$, $2$, $1$, $2$, and $1$, respectively, as we go over the five possibilities for $K_p$.

When $v$ is infinite, the lemma follows since $\R^3$ and $\R\oplus\C$ each have exactly one unramified quadratic extension, namely, $\R^6$ and $\R^2\oplus\C^2$, respectively. The lemma follows.
\end{proof}

For a prime $p$, let $T_p$ denote the set of all isomorphism classes of monogenized \'{e}tale cubic extensions of $\Q_p$, and let $\Sigma_p \subset T_p$ be an open and closed set whose boundary has measure $0$. We say that the family $\Sigma = (\Sigma_p)_p$ is a {\bf large collection} of local specifications if, for all but finitely many primes $p$, the set $\Sigma_p$ contains all pairs monogenized \'{e}tale cubic extensions of $\Q_p$ that are not totally ramified. We say that family of monogenized cubic number fields is a {\bf large family} if it is defined by a large collection of local specifications. 

Given a large family $\FF$ of monogenized cubic fields corresponding to the collection $\Sigma$ of local specifications, we associate to it the following large family $\EE$ of elliptic curves. For each finite prime $p$, define $\Lambda_p\subset\Z_p^2$ via
\begin{equation*}
\Lambda_p\defeq \{(I(f),J(f):f\in\Sigma_p\}.
\end{equation*}
We define $\Lambda_\infty^\pm$ to be the set of pairs $(I,J)\in\R^2$ such that $\pm\Delta(I,J)>0$ and $H(I,J)<X$. Define $\EE = \EE^\pm$ to be the family of elliptic curves associated to the collection $(\Lambda_v)_v$, where $\Lambda_\infty=\Lambda_\infty^\pm$. We take $\cS$ to be the $2$-Selmer structure for $\EE$, where $\cS_v$ is the unramified local condition at all $v$. Since $\FF$ is large, it immediately follows that $\EE$ is large. By definition, it follows that $\cS$ is acceptable. In the next crucial result, we will prove that $\cS$ on $\EE$ is also sub-soluble.
\begin{proposition}\label{propsubsol}
The Selmer structure $\cS$ on $\EE$ is sub-soluble.
\end{proposition}
\begin{proof}
Let $(K,\alpha)$ be a monogenized field in $\FF$, let $f$ be the corresponding monic binary form, and let $E=E_f$ be the associated elliptic curve. The content of the proposition is simply that every unramified class in $H^1(\Q_v,E[2])$ is soluble for every place $v$ of $\Q$. When $v$ is infinite, this is immediate since only the identity class is unramified. So assume that $v=p$ is a finite place.

We use the results of \cite{MR1370197} to understand the intersection $\cI_p$ of $H_\ur^1({\Q_v},E[2])$ with the image of the Kummer map in $H^1({\Q_v},E[2])$. Denote $\Q_p(E(\overline{\Q_p})[2])$ by $L$, and let $L'$ be the maximal subfield of $L$ that is unramified over $\Q_p$. Let $M$ denote the quadratic unramified extension of $L'$, and denote a generator of $\Gal(M/\Q_p)$ by $\tau$. Since Condition 4 in \cite[Lemma 5.1]{MR1370197} is easily verified, it follows from the same lemma that $E[2]$ is cohomologically trivial as a $\Gal(L/\Q_p)$-module. From the inflation-restriction exact sequence, it then follows that we have the isomorphism
\begin{equation*}
H^1(\Q_p,E[2])\simeq \Hom_{\Gal(L/K)}(\on{Gal}(\ol{\Q_p}/L),E[2]),
\end{equation*}
where $\Hom_{\Gal(L/K)}$ denotes homomorphisms invariant under the action of $\on{Gal}(L/K)$.
We then apply \cite[Theorem 3.4]{MR1370197} to deduce that we have the isomorphism
\begin{equation}\label{eqsubsolfin}
\cI_p\simeq \frac{E(M)[2]\cap (\tau-1)E(M)}{(\tau-1)E(M)[2]}.
\end{equation}
In what follows, we will prove that $\cI_p$ is isomorphic to $E(\Q_p)[2]$, which in conjunction with Lemma~\ref{lemunramsize}, yields the result.

Let $N$ denote the N\'{e}ron model of $E$ over $\cO_M$, the ring of integers of $M$. The identity component $N^0$ of $N$ is an open subgroup scheme of $N$; the group $N^0(\cO_M)$ is isomorphic to a subgroup $E_0(M)$ of $E(M)$; and $E_0(M)$ contains all the points in $E(M)$ with nonsingular reduction.
From \cite[Lemma 3.6]{MR1370197}, it follows that $E_0(M)\cap E(M)[2]$ is contained in $(\tau-1)E(M)$.
We now compute $\cI_p$ by considering each possibility for the \'etale algebra $K_p\defeq \Q_p[x]/(f(x,1))$. First suppose that $K_p$ is unramified. In this case, every element of $E(M)[2]$ has nonsingular reduction and thus belongs to $(\tau-1)E(M)$. When $K_p$ is $\Q_p^3$ or $\Q_{p^3}$, the action of $\tau$ on $E(M)[2]$ is trivial, and thus we have $\cI_p\simeq E(M)[2]=E(\Q_p)[2]$ as necessary. When $K_p=\Q_p\oplus\Q_{p^2}$, then $E(M)[2]$ is isomorphic to the Klein four group, and the action of $\tau$ interchanges two of the nontrivial elements, while keeping the other two elements fixed. Thus again, $\cI_p$ is isomorphic to $E(\Q_p)[2]$, as necessary.

Finally, we consider the cases where $K_p$ is ramified. In all of these cases, we have $E(M)[2]=E(\Q_p)[2]$, with a trivial action of $\tau$, so $\cI_p\simeq E(\Q_p)[2]\cap(\tau-1)E(M)=E(\Q_p)[2]$, since every element of $E(\Q_p)[2]$ has nonsingular reduction. This concludes the proof of Proposition \ref{propsubsol}.
\end{proof}

\subsection{The case of the narrow $2$-class group} \label{sec-narrowgen}

The difference between the ordinary and narrow class group is purely archimedean. Indeed, let $K$, $f$, and $E_f$ be as in Proposition~\ref{prop-classgroup}, and let $H^1_{\on{f-ur}}(\Q, E_f[2])$ be the subgroup of cohomology classes unramified at all finite places. Then class field theory yields the following analogue of Proposition~\ref{prop-classgroup} for the narrow $2$-class group:

\begin{proposition} \label{prop-narrowclassgroup}
    Let $\on{Cl}^{+,*}(K)$ denote the group dual to $\on{Cl}^+(K)$. We have a natural bijection
    $$\on{Cl}^{+,*}(K)[2] \overset{\sim}\longrightarrow H^1_{\on{f-ur}}(\Q,E_f[2]).$$
\end{proposition}

Of course, when $K$ is a complex cubic field, we have $\on{Cl}^+(K) \simeq \on{Cl}(K)$, so we may assume that $K$ is totally real, in which case a classical result of Armitage and Fr\"{o}hlich~\cite{MR214566} states that the kernel of the surjective map $\on{Cl}^{+}(K)[2] \twoheadrightarrow \on{Cl}(K)[2]$ has order $1$ or $2$. Now, by Proposition~\ref{prop-narrowclassgroup}, we get the following commutative diagram of group homomorphisms
\begin{equation*}
\begin{tikzcd}
\on{Cl}^{+,*}(K)[2] \arrow{r}{\sim} & H_{\on{f-ur}}^1(\Q, E_f[2]) \arrow{r} & H^1(\R, E_f[2]) \\
\on{Cl}^*(K)[2] \arrow[hook]{u} \arrow[swap]{r}{\sim} & H_{\on{ur}}^1(\Q, E_f[2]) \arrow[hook]{u} \arrow[two heads]{r} & 1 \arrow[hook]{u}
\end{tikzcd}
\end{equation*}
The result of Armitage-Fr\"{o}hlich implies that the image of $\on{Cl}^{+,*}(K)[2]$ in $H^1(\R, E_f[2])$ has order $1$ or $2$, since $\on{Cl}^*(K)[2]$ has trivial image in $H^1(\R, E_f[2])$, as noted in the proof of Proposition~\ref{propsubsol}. Now $H^1(\R, E_f[2]) \simeq (\Z/2\Z)^2$ by Lemma~\ref{lemEC2torGC}, so the image of $\on{Cl}^{+,*}(K)[2]$ in $H^1(\R, E_f[2])$ is one of the four subgroups of order at $1$ or $2$.

We now give a description of the image of $\on{Cl}^{+,*}(K)[2]$ in $H^1(\R, E_f[2])$ by leveraging a connection to the arithmetic of quartic number fields. A result of Bhargava--Hanke--Shankar (see~\cite[proof of Theorem~2.15]{BSHpreprint}, building on a result of Heilbronn~\cite{MR280461}) states that, if the Galois group of the normal closure of $K$ is isomorphic to $S_3$, then quadratic field extensions $\wt{K}/K$ unramified at all finite places are in bijection with \textbf{nowhere overramified} quartic fields $L$ with normal closure having Galois group $S_4$. Here, nowhere overramified means that no finite prime of $\Q$ has all of its ramification indices in $L$ divisible by $2$, and $K$ is none other than the cubic resolvent field of $L$. In particular, the discriminants of $L$ and $K$ are equal, so they must have the same sign. Thus, since we are assuming $K$ is totally real, $L$ is either totally real or totally complex.

Next, a theorem of Wood~\cite[Theorem~1.1]{MR2899953} implies that the ring of integers of a quartic number field with monogenic cubic resolvent can always be realized as the ring of global sections of a degree-$4$ subscheme of $\mathbb{P}_{\Z}^1$, i.e., the vanishing locus of an integral binary quartic form. More precisely, let $V_2$ denote the space of binary quartic forms, as in the proof of Theorem~\ref{thm-countingselmers}. Then by Wood's result, the set of nowhere overramified quartic fields $L$ with a monogenized cubic resolvent field $K$ are in bijection with the set of $\on{PGL}_2(\Z)$-orbits on integral binary quartic forms $g \in V_2(\Z)^{(I(f),J(f))}$ such that $g$ is nowhere overramified, meaning that it does not factor as a scalar multiple of the square of a binary quadratic form modulo any finite prime of $\Q$.

Synthesizing the discussion of the above three paragraphs, we obtain a bijective map
\begin{equation} \label{eq-binaryquartics}
H^1_{\on{f-ur}}(\Q, E_f[2]) \overset{\sim}\longrightarrow \on{PGL}_2(\Z) \backslash \left\{g \in  V_2(\Z)^{(I(f),J(f))} : g \text{ nowh.~overram.}\right\}.
\end{equation}
By work of Bhargava--Gross~\cite[\S6]{MR3156850}, a class $\xi \in H^1_{\on{f-ur}}(\Q, E_f[2])$ is soluble in $H^1(\R, E_f[2])$ if and only if the genus-$1$ curve with equation $z^2 = g(x,y)$ has an $\R$-point, where $g$ is a binary quartic form corresponding to $\xi$ via the bijection in~\eqref{eq-binaryquartics}. Thus, it might happen that $\xi$ corresponds to a negative-definite binary quartic form, in which case $\xi$ is insoluble over $\R$, and we cannot na\"{i}vely realize $\on{Cl}^{+,*}(K)$ as a Selmer group associated to a subsoluble $2$-Selmer structure on $E_f$.

On the other hand, note that if $g$ is negative-definite, then $-g$ is positive-definite, so the curve $z^2 = g(x,y)$ arising from a class in $H^1_{\on{f-ur}}(\Q, E_f[2])$ is always soluble over $\R$ up to quadratic twist. Letting $f^*(x,y) = f(x,-y)$, we get a commutative diagram
\begin{equation} \label{eq-binaryquartics2}
\begin{tikzcd}
H^1(\R, E_f[2]) \arrow[swap]{d}{\rotatebox{90}{$\sim$}}  & H^1_{\on{f-ur}}(\Q, E_f[2]) \arrow{l}\arrow{r}{\sim} \arrow[swap]{d}{\rotatebox{90}{$\sim$}} & \on{PGL}_2(\Z) \backslash \left\{g \in  V_2(\Z)^{(I(f),J(f))}  : g \text{ nowh.~overram.}\right\} \arrow{d}{g \,\mapsto\, -g} \\
H^1(\R, E_{f^*}[2])  & H^1_{\on{f-ur}}(\Q, E_{f^*}[2]) \arrow{l}\arrow[swap]{r}{\sim} & \on{PGL}_2(\Z) \backslash \left\{g \in  V_2(\Z)^{(I(f),-J(f))} : g \text{ nowh.~overram.}\right\} 
\end{tikzcd}
\end{equation}
If we write $H^1(\R, E_f[2]) = \{1, \xi_1, \xi_2, \xi_3\} = H^1(\R, E_{f^*}[2])$, then the vertical isomorphism on the left of~\eqref{eq-binaryquartics2} sends $1 \mapsto 1$, and without loss of generality exchanges $\xi_1$ and $\xi_2$ but fixes $\xi_3$. The image of $H^1_{\on{f-ur}}(\Q, E_f[2])$ in $H^1(\R, E_f[2])$ is thus either $\{1\}$, $\{1, \xi_1\}$, or $\{1, \xi_2\}$. We may assume that $\xi_1$ is soluble for $E_f$, in which case $\xi_2$ is soluble for $E_{f^*}$. In particular, $\xi \in H^1_{\on{f-ur}}(\Q, E_f[2])$ is trivial or insoluble over $\R$ if and only if the corresponding class $\xi \in H^1_{\on{f-ur}}(\Q, E_{f^*}[2])$ is soluble.

Now, let $\mc{S}(E_f)$ be the unramified Selmer structure as in the previous section, and let $\wt{\mc{S}}(E_f)$ be the $2$-Selmer structure defined by taking the soluble condition at $\infty$ and the unramified condition at all finite primes; note that $\wt{\mc{S}}$ is sub-soluble. Then combining the above discussion with Proposition~\ref{prop-narrowclassgroup} and the principle of inclusion-exclusion yields the following result:
\begin{proposition} \label{prop-narrowsub}
Let $K$, $f$, and $E_f$ be as above. Then we have that
\begin{align*}
& \on{Cl}_2^+(K) \simeq H^1_{\on{f-ur}}(\Q, E_f[2]) = \on{Sel}_{\wt{\mc{S}}(E_f)}(E_f) \cup \on{Sel}_{\wt{\mc{S}}(E_f{^*})}(E_{f^*}), \quad \text{and} \\
& \on{Sel}_{\wt{S}(E_f)}(E_f) = \on{Sel}_{\mc{S}(E_f)}(E_f) \quad \text{or}\quad \on{Sel}_{\wt{\mc{S}}(E_f{^*})}(E_{f^*}) = \on{Sel}_{\mc{S}(E_f)}(E_f).
\end{align*}
In particular, we have that
\begin{align} \label{eq-inclusionexclusion1}
\#\on{Cl}_2^+(K) & = \big(\#\on{Sel}_{\wt{\mc{S}}(E_f)}(E_f) - \#\on{Sel}_{\mc{S}(E_f)}(E_f)\big) +  \#\on{Sel}_{\wt{\mc{S}}(E_{f^*})}(E_{f^*})  \\
\#\on{Cl}_2^+(K)^2 & = \big(\#\on{Sel}_{\wt{\mc{S}}(E_f)}(E_f)^2 - \#\on{Sel}_{\mc{S}(E_F)}(E_f)^2\big) +  \#\on{Sel}_{\wt{\mc{S}}(E_{f^*})}(E_{f^*})^2.  \label{eq-inclusionexclusion2}
\end{align}
\end{proposition}
In this way, the size of the narrow $2$-class group, as well as its square, can be expressed in terms of sizes of Selmer groups associated to sub-soluble Selmer structures.

\subsection{The proof of Theorem~\ref{thm-main}(a) and (b)} \label{sec-proofab}

Let the large family $\mc{E} = \mc{E}^{\pm}$ and the Selmer structure $\mc{S}$ be defined just as in \S\ref{sec-ordinary}. By Theorem~\ref{thm-countingselmers} in conjunction with Propositions~\ref{prop-classgroup} and~\ref{propsubsol}, it suffices to compute the local masses $\mc{M}_v(2; \mc{E},\mc{S})$ for each place $v$ of $\Q$.


For the mass at $\infty$, substituting the result of Lemma~\ref{lemunramsize} into~\eqref{eqMass} and observing that $\#E(\R)[2] = 4$ if $E \in \mc{E}^+$ and $\#E(\R)[2] = 2$ if $E \in \mc{E}^-$, we deduce that
\begin{equation} \label{eq-infmass}
    \mc{M}_\infty(2; \mc{E},\mc{S}) = \begin{cases} \frac{1}{4}, & \text{ if $\mc{E} = \mc{E}^+$,} \\
\frac{1}{2}, & \text{ if $\mc{E} = \mc{E}^-$.} \end{cases}
\end{equation}
 As for the masses at finite places $p$, substituting the result of Lemma~\ref{lemunramsize} into~\eqref{eqMass}, we deduce that
 \begin{equation} \label{eq-pmass}
    \mc{M}_p(2; \mc{E},\mc{S}) = 1.
\end{equation}
Substituting the masses~\eqref{eq-infmass} and~\eqref{eq-pmass} into the third displayed equation in Theorem~\ref{thm-countingselmers} completes the proof of Theorem~\ref{thm-main}(a) and (b).

\subsection{The proof of Theorem~\ref{thm-main}(c)} \label{sec-proofc}

Let the large family $\mc{E} = \mc{E}^+$ be defined as in \S\ref{sec-ordinary}, and let $\mc{S}$ and $\wt{\mc{S}}$ be the Selmer structures defined in \S\ref{sec-ordinary} and \S\ref{sec-narrowgen}, respectively. By Theorem~\ref{thm-countingselmers} and~\eqref{eq-slightgen} in conjunction with Propositions~\ref{prop-narrowclassgroup} and~\ref{prop-narrowsub}, it suffices to compute the local masses $\mc{M}_v(2; \mc{E}, \mc{S})$ and $\mc{M}_v(2; \mc{E}; \wt{\mc{S}})$ for each place $v$ of $\Q$. The first family of masses was computed in the previous section; as for the second family, we have at the infinite place that 
\begin{equation} \label{eq-infmassc}
\mc{M}_\infty(2; \mc{E}; \wt{\mc{S}}) = 2 \times \mc{M}_\infty(2; \mc{E}; \mc{S}) = 1/2
\end{equation}
and for each finite place $p$ that
\begin{equation} \label{eq-pmassc}
\mc{M}_p(2; \mc{E}; \wt{\mc{S}}) = \mc{M}_p(2; \mc{E}; \mc{S}) = 1
\end{equation}
Substituting the masses~\eqref{eq-infmassc} and~\eqref{eq-pmassc} into~\eqref{eq-slightgen}, as well as into the third displayed equation in Theorem~\ref{thm-countingselmers}, and applying Proposition~\ref{prop-narrowsub} completes the proof of Theorem~\ref{thm-main}(c).

\subsection*{Acknowledgments}

We are very grateful to Aaron Landesman, Peter Sarnak, Ari Shnidman, Artane Siad, Melanie Matchett Wood, and Don Zagier for helpful conversations. The first-named author was supported by a Simons Investigator Grant and NSF grant~DMS-1001828. The second-named author was supported by an NSERC discovery grant and Sloan fellowship. The third-named author was supported by the NSF, under the Graduate Research Fellowship as well as Award No.~2202839.

\bibliographystyle{abbrv} \bibliography{bibfile}
\end{document}